\documentclass[a4paper,12pt]{amsart}

\usepackage{a4wide}
\usepackage{amsmath,amsthm,latexsym,amssymb,pgf,enumitem,hyperref} 
\usepackage{tikz,subcaption,graphicx} 
\usepackage{multirow}
\usetikzlibrary{arrows,decorations.pathmorphing,backgrounds,positioning,fit,petri,calc}

\newtheorem{claim}{Claim}
\newtheorem{lemma}{Lemma}

\newtheorem{theorem}[lemma]{Theorem}

\theoremstyle{remark}

\newtheorem{conjecture}{Conjecture}

\newtheorem{definition}{Definition}
\newcommand{\old}[1]{{}}
\usepackage[notref,notcite, final]{showkeys}

\title{Maximum diameter of $3$- and $4$-colorable graphs}  

\author{\'Eva Czabarka}
\author{Stephen J. Smith}
\author{L\'aszl\'o Sz\'ekely}
\address{\'Eva Czabarka\\Department of Mathematics \\ University of South Carolina \\ Columbia SC 29212 \\ USA}
\email{czabarka@math.sc.edu}
\address{Stephen J. Smith\\Department of Mathematics\\ University of South Carolina \\ Columbia SC 29212 \\ USA}
\email{sjs8@email.sc.edu}
\address{L\'aszl\'o Sz\'ekely\\Department of Mathematics \\ University of South Carolina \\ Columbia SC 29212 \\ USA}
\email{szekely@math.sc.edu}


\begin{document}
\begin{abstract} P. Erd\H{o}s, J. Pach, R. Pollack, and Z. Tuza
   [\textit{ J. Combin. Theory} {\bf B 47} (1989), 279--285] made conjectures for the maximum diameter of connected graphs without a 
   complete subgraph $K_{k+1}$, which have
   order $n$ and minimum degree $\delta$. Settling a weaker version of a problem, by strengthening the  $K_{k+1}$-free  condition to $k$-colorable,
   we solve the problem for $k=3$ and $k=4$ using a  unified linear programming duality approach. The case $k=4$ is a substantial simplification 
   of the result of  \'E. Czabarka, P. Dankelmann, and L. A. Sz\'ekely
    [\textit{Europ. J. Comb.} {\bf 30} (2009), 1082--1089].
\end{abstract}
\maketitle

\section{Introduction}

We study the maximum diameter of connected graphs in terms of other graph parameters such as order, minimum degree, etc.
Several papers \cite{Amar, EPPT, Gold, Moon} have shown that:

\begin{theorem}\label{th:ori}
For a fixed minimum degree $\delta \geq 2$,  every  connected graph $G$ of order $n$ satisfies 
$\operatorname{diam}(G) \leq \frac{3n}{\delta+1}+O(1)$, as  $n\rightarrow\infty$.
\end{theorem} 
This upper bound is sharp (even for $\delta$-regular graphs \cite{Smyth}), but the constructions have complete subgraphs, whose order increases with $n$.  
Erd\H{o}s, Pach, Pollack, and Tuza  \cite{EPPT} conjectured that the  upper bound in Theorem~\ref{th:ori} can be improved, if large cliques are excluded:
\begin{conjecture}[\cite{EPPT}]
\label{con:Erdosetal}
Let $r,\delta\geq 2$ be fixed integers and let $G$ be a connected graph of order $n$ and minimum degree $\delta$. 
\begin{enumerate}[label={\upshape (\roman*)}]
\item\label{conpart:even} If $G$ is $K_{2r}$-free and $\delta$ is a multiple of
$(r-1)(3r+2)$ then, as $n\rightarrow \infty$,
\begin{eqnarray*}
\operatorname{diam}(G) &\leq& \frac{2(r-1)(3r+2)}{(2r^2-1)}\cdot \frac{n}{\delta} + O(1)\\
&=&\left(3-\frac{2}{2r-1}-\frac{1}{(2r-1)(2r^2-1)}\right)\frac{n}{\delta}+O(1). 
\end{eqnarray*}
\item\label{conpart:odd}  If $G$ is $K_{2r+1}$-free and $\delta$ is a multiple of $3r-1$, then, as $n\rightarrow \infty$,
\[ \operatorname{diam}(G) \leq \frac{3r-1}{r}\cdot \frac{n}{\delta} + O(1)=\left(3-\frac{2}{2r}\right)\frac{n}{\delta}+O(1). \]
\end{enumerate}
\end{conjecture} 
Furthermore, they created examples showing that the above conjecture, if true, is sharp, and showed
part (ii) of the conjecture  for $r=1$.

Czabarka, Dankelmann and Sz\'ekely \cite{dankelmanos} arrived at the
conclusion of Conjecture \ref{con:Erdosetal}~\ref{conpart:odd} for $r=2$ under a stronger hypothesis:
\begin{theorem} 
\label{th:CDS}
For every connected $4$-colorable graph $G$ of order $n$ and
minimum degree $\delta\ge 1$,
\( \operatorname{diam}(G) \leq \frac{5n}{2\delta}-1. \)
\end{theorem}

Czabarka, Singgih and Sz\'ekely  \cite{counterexpaper} gave an infinite family of $(2r-1)$-colorable (hence $K_{2r}$-free) graphs with diameter
$\frac{(6r-5)(n-2)}{(2r-1)\delta+2r-3}-1$, providing a counterexample for  Conjecture \ref{con:Erdosetal}~\ref{conpart:even} for
every  $r\geq 2$ and $\delta> 2 (r-1)(3r+2)(2r-3)$.  
The question whether Conjecture \ref{con:Erdosetal}~\ref{conpart:even} holds in the 
range $(r-1)(3r+2)\le\delta\le 2(r-1)(3r+2)(2r-3)$ remains open.
The counterexample led Czabarka {\sl et al.} \cite{counterexpaper} 
 to the modified conjecture below, which no longer requires cases for the parity of the order of the excluded complete subgraphs:
 \begin{conjecture}[\cite{counterexpaper}] \label{con:7o3}
 For every $k\ge 3$ and $\delta\ge \lceil\frac{3k}{2}\rceil-1$, 
if $G$ is a $K_{k+1}$-free (under a stronger hypothesis, $k$-colorable) connected graph of order $n$ and  minimum degree at least $\delta$,  
  $\operatorname{diam}(G)\leq \left(3-\frac{2}{k}\right)\frac{n}{\delta}+O(1)$. 
  \end{conjecture}
 Czabarka, Singgih and Sz\'ekely  \cite{kcolorable}  showed that the extremal graphs for the diameter maximization problem of Conjecture~\ref{con:7o3}
include graphs blown up from some very specific structures, called {\sl canonical clump graphs}. Furthermore, \cite{kcolorable}  showed using the
weak duality theorem of  linear programming that providing a sufficiently good solution for a dual problem on canonical clump graphs gives an upper bound for the 
diameter of graphs blown up from canonical clump graphs (see Theorem~\ref{th:tool}), and hence an upper bound for the diameter maximization problem of Conjecture~\ref{con:7o3}.     
Using this method, they proved:
\begin{theorem} [\cite{{kcolorable}}] \label{th:upperbound}        Assume $k\ge 3$. If $G$ is a connected $k$-colorable graph of minimum degree at least $\delta$, then
\[ \operatorname{diam}(G)\leq \frac{3k-4}{k-1}\cdot\frac{n}{\delta}-1=\left(3-\frac{1}{k-1}\right)\frac{n}{\delta}-1\].
\end{theorem}
\noindent    Czabarka, Singgih and Sz\'ekely  \cite{kcolorable}  also made a slight improvement on Theorem~\ref{th:upperbound}   for $3$-colorable graphs, but with a different argument.

 In this paper we give a common short proof of the Conjecture~\ref{con:7o3} (under the stronger hypothesis) for $k=3$ and $4$  (the latter  being  Theorem~\ref{th:CDS}  from  Dankelmann {\sl et al.} \cite{dankelmanos})  using the approach  above: 
 \begin{theorem}\label{th:main} Assume $k=3$ or $4$. If $G$ is a connected $k$-colorable graph of order $n$, and  of minimum degree at least $\delta\ge 1$, then $\operatorname{diam}(G)\le\left(3-\frac{2}{k}\right)\frac{n}{\delta}-1.$
\end{theorem}
 
The main tool of the proof is still the use of canonical clump graphs, however, we focus on an even smaller class, {\sl strongly canonical clump graphs}, of which blown up copies
are still present among the extremal graphs for the diameter maximization problem of Conjecture~\ref{con:7o3}, as shown in Section~\ref{sec:clump}.
We partition the strongly canonical graph into segments of three types. Weighting the vertices such that the total weight of the neighbors of any vertex is at most $1$ 
and the average weight of a layer in each segment is $\frac{k}{3k-2}$ finishes the proof. 
When $k\in\{3,4\}$, Type 1 and Type 2 segments (defined in Section \ref{sec:definitions}) have a very limited structure, as shown in Lemma~\ref{lm:main}.
For $k\ge 5$ we have examples of segments that cannot be weighted according to this scheme, so new ideas are needed.

\section{Clump Graphs}\label{sec:clump}

Given a $k$-colorable connected graph $G$ of order $n$ and minimum degree at least $\delta$, choose a vertex $x$  whose  eccentricity
is $\operatorname{diam}(G)$.
Take  a \emph{fixed} good $k$-coloring of $G$. Let \emph{layer} $L_i$ denote the set of vertices at distance $i$ from $x$, and a \emph{clump} in $L_i$ be the set of vertices in $L_i$ that have the same color. 
The number of layers is  $\operatorname{diam}(G)+1$. We call a graph \emph{layered}, if such a vertex $x$ and the distance layers $L_0=\{x\},L_1,\ldots, L_D$ are given.
Let $c(i) \in \{1,2,\ldots, k\}$ denote the number of colors used in layer $L_i$ by our fixed coloration. We can assume without loss of generality that
 any two vertices in layer $L_i$   in $G$, which are differently colored, are joined by an edge in $G$, and also that two vertices in consecutive layers, which are differently colored, are also joined by an edge in $G$. We call this assumption \emph{saturation}  with respect to the fixed good $k$-coloring. 
Assuming saturation does not make loss of generality, as  adding these edges does not decrease degrees, keeps the fixed good $k$-coloration, and does not reduce
the diameter, while making the graph more structured for our convenience. 

From the layered and saturated  graph $G$ above, 
we create an 
\emph{unweighted clump graph} $H=H(G)$.  Vertices of $H$ correspond to the clumps of $G$. Two vertices of $H$ are connected by an edge if there were edges between the corresponding clumps in $G$. $H$ is naturally $k$-colored and layered, based on the coloration and layering of $G$.  With a slight abuse of notation, we denote the layers of $H$ by
$L_i$ as well. 
To create a \emph{weighted clump graph}, 
we assign positive integer \emph{weights} to each vertex of the unweighted clump graph.  Blowing up vertices of $H$ into as many copies as their weight is, we obtain a  bigger  $k$-colorable graph
of the same diameter (we do not put edges between successors of the same vertex). In case the weights are the cardinalities of the clumps in $G$, after the blow-up of $H=H(G)$ we get back $G$. The degree of a vertex  $v$ in a blow-up
of $H$, where $v$ is a successor of a vertex $w$ of $H$ by blow-up, is the sum of the weights of neighbors of the vertex $w$ in $H$.   The number of vertices in a blow-up of $H$  is the sum of the weights of all vertices in $H$.

The following theorem was proven in \cite{kcolorable}:  
\begin{theorem}[\cite{kcolorable}]  \label{th:canonical} Assume $k\ge 3$. Let $G'$ be a $k$-colorable connected  graph of order $n$, diameter $D$ and minimum degree at least $\delta$. 
Then there is a saturated $k$-colored 
and layered connected graph $G$ of the same parameters $n$ and $\delta$, with layers $L_0,\ldots,L_D$, for which  the following hold for every
$i$ $(0\le i\le D-1)$:
\begin{enumerate}[label={\upshape (\roman*)}]
\item\label{part:1k} If  $c(i)=1$, then $c(i+1)\le k-1$.
\item\label{part:manycolor} The number of colors used to color the set $L_i\cup L_{i+1}$ is $\min(k,c(i)+c(i+1))$. In particular, when
$c(i)+c(i+1)\le k$, then $L_i$ and $L_{i+1}$ do not share any color.
\item\label{part:k1} If $c(i)=k$, then $i\ge 2$ and  $c(i+1)\ge 2$.
\item\label{part:singleton} If $|L_i|>c(i)$, i.e., $L_i$ contains two vertices of the same color, then $i>0$ and $c(i)+\max\bigl(c(i-1),c(i+1)\bigl)\geq k$.
\end{enumerate}
\end{theorem}
{\sl Canonical clump graphs} were defined in \cite{kcolorable} as $H=H(G)$ clump graphs, where $G$ satisfies the conclusions of Theorem~\ref{th:canonical}. Now we define
{\sl strongly canonical clump graphs}   for $D\geq 2$ as $H=H(G)$  canonical clump graphs (i.e., $G$ satisfies the conclusions of Theorem~\ref{th:canonical}), and in addition,  
$c(0)=c(D)=1$.

It is not difficult to see  the following: if the graph $G'$ in the assumption of Theorem~\ref{th:canonical} 
 is layered with $|L_0'|=1$ and   $c'(D)=1$,   
 then the proof of Theorem~\ref{th:canonical}     in \cite{kcolorable} provides a layered $G$ with  $|L_0|=1$  (and hence $c(0)=1$),  and $c(D)=1$.
Based on this observation, the following lemma implies that to resolve Conjecture~\ref{con:7o3} (or proving Theorem~\ref{th:main}), we may assume that $G$ has a 
strongly canonical clump graph.

\begin{lemma}\label{lm:cD1} Assume $k\ge 3$ and $D\geq 2$. Let $G'$ be a $k$-colored layered connected  graph of order $n$, diameter $D$, and minimum degree at least $\delta$,
with layers $L'_0,\ldots,L'_D$. 
Then 
 there is a $k$-colored  layered connected graph $G$ of the same parameters, with layers $L_0,\ldots,L_D$, for which $c(0)=c(D)=1$, and for each
$i$ $(0\le i\le D-2)$, we have  $c'(i)=c(i)$ and $L'_i=L_i$.
\end{lemma}

\begin{proof}
As $|L'_0|=1$ is necessary in a layered graph, we must have $c'(0)=1$, and if $c'(D)=1$, the choice $L'_i=L_i$ suffices. If $c'(D)>1$, pick a color $A$ in $L_D'$. If possible, pick such a color  
that also appears in $L_{D-2}'$. This ensures that 
 for all colors $B$ in $L_D'$ such that $B\ne A$ there is a color $C$ in $L_{D-2}'$ such that $B\ne C$ (where $C=A$, if $A$ appeared in $L_{D-2}$, otherwise any color in $L_{D-2}$ works). Create a layered graph graph $G$ from $G'$ by moving all vertices in $L_{D}'$
that are not colored $A$ to the next-to-last layer, which will be $L_{D-1}$, and connect them to all vertices in $L_{D-2} = L_{D-2}'$ that have different color. Note that 
for all vertices of $ L_{D-1}$,  there is at least one such vertex.
As we only changed the number of vertices in layers $D-1$ and $D$, and did not change the coloration of the vertices, the claim follows.
\end{proof}

\section{Duality}
Let $\mathcal{H}_{k,D,\delta}$   denote the family of unweighted canonical clump graphs of 
diameter $D$ that arises from connected $k$-colorable graphs $G$  with diameter $D$ and minimum degree at least $\delta$, when the order of $G$ is unspecified. 
We will rely on the following result from~\cite{kcolorable}:
\begin{theorem}
\label{th:tool} {\rm  (\cite{kcolorable})}
Fix $k\ge 3$. Assume that there exists constants $\tilde{u}>0$ and $C\ge 0$ such that for all $D$ and $\delta$, and for all $H\in\mathcal{H}_{k,D,\delta}$, the optimum of the linear program
\[
\text{Maximize  } \delta\cdot \sum_{y\in V(H)} u(y), 
\]
subject to the condition
\begin{equation}
\forall x\in V(H) \,\,\,\,\,\,\,\, \sum_{y\in V(H)\, : \, xy\in E(H)} u(y)\le 1. \label{dualcond}
\end{equation}
is at least
\[
\tilde{u}\delta D+C.
\]
Then for any $k$-colorable graphs $G$ with minimum degree $\delta$ on $n$ vertices, we have
\[
\operatorname{diam}(G) \le \frac{1}{\tilde{u}}\frac{n}{\delta}-\frac{C}{\tilde u}.
\]
\end{theorem}

In  Theorem~\ref{th:tool}  and in its proof we may change the family of canonical clump graphs $\mathcal{H}$ to the family of strongly canonical clump graphs $\mathcal{H'}$
keeping all arguments valid.

\section{Some definitions and observations} \label{sec:definitions}
Recall that we use the sloppy notation $L_i$ for the layers of the clump graph $H(G)$ as well, not just for the layers of $G$. Hence $c(i)=|L_i|$, if $L_i$ denotes a layer
of the clump graph. Based on the arguments of Section~\ref{sec:clump}, we have:

\begin{claim}\label{def:strongcan} An unweighted $k$-colorable strongly canonical clump graph with layers $L_0,\ldots,L_D$   satisfies the following properties:
\begin{enumerate}[label={\upshape (\roman*)}]
\item\label{part:ends} $|L_0|=|L_D|=1$,
\item\label{part:nok1} If $|L_i|=k$, then $2\le i \le D-1$ and $\min(|L_{i-1}|,|L_{i+1}|)\ge 2$, and
\item\label{part:match} For $i\in[D]$, the edges that do not appear between $L_{i-1}$ and $L_i$ form a matching of size $\max(k,|L_{i-1}|+|L_i|)-k$.
\end{enumerate}
\end{claim}

For the following definition, and also for the rest of this section, assume that we are given
a $k$-colorable canonical clump graph $H$ with layers $L_0,\ldots ,L_D$. We define for convenience two additional layers, as  $L_{-1}=L_{D+1}=\emptyset$.
For a vertex $x\in V(H)$, let $N(x)$ denote the set of neighbors of $x$.
\begin{definition} For each $i:0\le i\le D$, define the set $S_i=\{x\in L_i: L_{i-1}\cup L_i\subseteq N(x)\}$. We call a
layer $L_i$ \emph{big} if $|S_i|>\frac{k}{2}$. A layer is \emph{small} if it is not big.
\end{definition}

Note that if $L_i$ is big, then $i\notin\{0,D\}$. We set $S_{-1}=S_{D+1}=\emptyset$, in accordance with $L_{-1}=L_{D+1}=\emptyset$.

\begin{lemma}\label{lm:basic} Assume $D\geq 2$. Let $H$ be an unweighted $k$-colorable strongly caonical clump graph with layers $L_0,\ldots,L_D$.
 The following is true for each $i:0\le i\le D$:
\begin{enumerate}[label={\upshape (\roman*)}]
\item\label{part:sizes} $|L_i|\le k-\max(|S_{i-1}|,|S_{i+1}|)$,
\item\label{part:ub} $|S_i|\le k-1$,
\item\label{part:type} if $L_i$ is big, then $1\le i\le D-1$ and $L_{i-1},L_{i+1}$ are small,
\item\label{part:small} if $|L_i|=1$, then $L_i=S_i$,
\item\label{part:type3} $\max(|L_i\setminus S_i|,|L_{i+1}\setminus S_{i+1}|)\le k-|S_i|-|S_{i+1}|$,
\item\label{part:big}  if $|S_i|=k-1$, then $L_i=S_i$ and for $j=i\pm 1$, $|L_j|=|S_j|=1$,
\item\label{part:last} if $k\in\{3,4\}$ and $L_i$ is big, then $|S_i|=k-1$.
\end{enumerate}
\end{lemma}

\begin{proof}
\ref{part:sizes} follows from the facts that $S_{i-1}\cup L_i$, and also  $S_{i+1}\cup L_i$, forms a complete subgraph in $k$-colorable graph $H$. 

\ref{part:ub} follows from \ref{part:sizes} and the fact that $L_{i-1}\cup L_{i+1}$ contains at least one vertex.

\ref{part:type} follows from \ref{part:sizes}.

\ref{part:small} :
As $|L_i|=1$, then from Claim~\ref{def:strongcan}~\ref{part:nok1} we get that $\max(|L_{i-1}|,|L_{i+1}|)\le k-1$. By Claim~\ref{def:strongcan}~\ref{part:match},  the vertex in $L_i$ is adjacent to every vertex in $L_{i-1}\cup L_{i+1}$. 

For \ref{part:type3}, 
$S_{i}\cup S_{i+1}\cup (L_i\setminus S_i)$ forms a complete graph in the $k$-colorable graph $H$, and hence $|L_i\setminus S_i|\leq k-|S_i|-|S_{i+1}|$, and similarly,
$S_{i}\cup S_{i+1}\cup (L_{i+1}\setminus S_{i+1})$ forms a complete graph, and hence  $|L_{i+1}\setminus S_{i+1}|\leq k-|S_i|-|S_{i+1}|$.

For  \ref{part:big},
if $|S_i|=k-1$, then $1\leq i\leq D-1$. By~\ref{part:sizes}, $|L_{i-1}|=|L_{i+1}|=1$, and by Claim~\ref{def:strongcan}~\ref{part:nok1},  $|L_i|\le k-1$, and by $k-1=|S_i|\leq |L_i|\le k-1$, $S_i=L_i$.

For \ref{part:last},
if $L_i$ is big, then by definition $\frac{k}{2}<|S_i|$. By \ref{part:ub} $|S_i|\le k-1$. For $k\in\{3,4\}$, these give $|S_i|=k-1$.
\end{proof}

\begin{definition} Let $H$ be an unweighted $k$-colorable strongly canonical clump graph with layers $L_0,\ldots,L_D$. If for some  $s\geq 1$  the contiguous segment of layers
$L_i,L_{i+1},\ldots,L_{i+2s}$ satisfies all the following conditions:
\begin{enumerate}[label={\upshape (\roman*)}]
\item for each $j:1\le j\le s$  the layer $L_{i+2j-1}$ is big (thus, $L_{i+2j-2},L_{i+2j}$ are small),
\item $i= 0$ or $L_{i-1}$ is small,
\item $i+2s= D$ or $L_{i+2s+1}$ is small,
\end{enumerate}
then we say that the contiguous segment is Type 1, if $s=1$, and Type 2, if $s>1$.
\end{definition}

\begin{definition} Let $H$ be an unweighted $k$-colorable strongly canonical clump graph with layers $L_0,\ldots,L_D$. Assume that $t\geq 0$. We say that the contiguous segment of layers
$L_i,L_{i+1},\ldots,L_{i+t}$ is Type 3, if the following hold:
\begin{enumerate}[label={\upshape (\roman*)}]
\item for each $j:i\le j\le i+t$ the layer $L_{j}$ is small,
\item if $i\ne 0$ then $i> 2$ and $L_{i-2}$ is big (thus, $L_{i-1},L_{i-3}$ are small),
\item if $i+t\ne D$ then $i+t<D-2$ and $L_{i+t+2}$ is big (thus, $L_{i+t+1},L_{i+t+3}$ are small).
\end{enumerate}
\end{definition}
Observe that in a Type 3 segment every layer is small.

The following Lemma easily follows from the definition of strongly canonical clump graphs and Lemma~\ref{lm:basic}.
\begin{lemma}\label{lm:main}
Let $H$ be an unweighted $k$-colorable strongly canonical clump graph. Then the layers $L_{0},\ldots,L_D$ can be partitioned into segments of Type 1, Type 2 and Type 3. Moreover, 
if $k\in\{3,4\}$ and $L_j$ is a layer in a Type 1 or Type 2 segment, then $L_j=S_j$ and $|L_j|\in\{1,k-1\}$.
\end{lemma}

\section{Proof of Theorem~\ref{th:main}}

Assume $k\in\{3,4\}$, and let $H$ be an unweighted $k$-colorable strongly canonical clump graph. By Theorem~\ref{th:tool}, it is enough to find a dual weighting $u$ of the vertices of $H$, which satisfies the conditions of that theorem and has total weight 
$(D+1)\frac{k}{3k-2}$. Fix a partition of the layers of $H$ into segments of Type 1, Type 2 and Type 3 according to Lemma~\ref{lm:main}.
 For shortness, we will say that layer $L_i$ is of Type $j$,  if $L_i$ falls into a segment  of Type $j$.

Consider a vertex $v$ in a layer $L_i$. Set $u(v)$ as follows:
\begin{itemize}
\item If $L_i$ is of Type 1:
$u(v)=\begin{cases}  \frac{2}{3k-2}, & {\rm if }  |L_i|=k-1, \\  \frac{k+2}{2(3k-2)}  & {\rm otherwise}. \end{cases}$
\item
If $L_i$ is of Type 2:\\
$u(v)=\begin{cases} \frac{1}{2(k-1)} ,& {\rm if\ } |L_i|=k-1,\\
\frac{k+2}{2(3k-2)}, & {\rm if\ } |L_i|=1 \hbox{ and } L_i  \hbox{\rm \  is not the first or last layer of the segment,}\\
\frac{3k+2}{4(3k-2)}, & \hbox{otherwise}.
\end{cases}$\\
Note that as $k\ge 3$, $\frac{k+2}{2(3k-2)}<\frac{3k+2}{4(3k-2)}<\frac{k}{3k-2}$.
\item
If $L_i$ is of Type 3:\\
$u(v)=\begin{cases} 
\frac{k}{(3k-2)|L_i|}, &{\rm if\ } |L_i|\le\frac{k}{2}\\
\frac{2}{3k-2}  ,& {\rm if\ } |L_i|>\frac{k}{2} {\rm \ and\ } v\in S_i\\
\frac{k-2|S_i|}{(3k-2)(|L_i|-|S_i|)} , & \hbox{otherwise}.\end{cases}$\\
 Note that as $|S_i|\le\frac{k}{2}$, we get  $u(v)\ge 0$. Also, $u(v)\geq \frac{2}{3k-2}$ if $|L_i|\le\frac{k}{2}$ or $v\in S_i$. 
\end{itemize}

We define the weight $u(X)$ of a vertex set $X$ as $\sum_{v\in X} u(v)$.
It is easy to check that for any Type 3 layer $L_i$, $u(L_i)=\frac{k}{3k-2}$. Also, first and last layers in a segment of any type have weight at most $\frac{k}{3k-2}$.

If $L_i,L_{i+1},L_{i+2}$ is a Type 1 segment, then $u(L_i)+u( L_{i+1})+u(L_{i+2})=(k-1)\cdot \frac{2}{3k-2}+2\cdot\frac{k+2}{2(3k-2)}=3\cdot\frac{k}{3k-2}$.

Assume that $L_i,L_{i+1},\ldots,L_{i+2s}$ is a Type 2 segment. The total weight of the  layers of this segment is
$
s\cdot \frac{1}{2}+ (s-1)\cdot \frac{k+2}{2(3k-2)}+2\cdot \frac{3k+2}{4(3k-2)}=(2s+1)\cdot\frac{k}{3k-2}
$.

This gives that the total weight of $H$ is $(D+1)\frac{k}{3k-2}$, as required. We need to check condition (\ref{dualcond}) at every vertex $v\in V(H)$.

Assume first that $v\in L_i$, where $L_i$ is of Type 1.

If $L_i$ is big, then $u(N(v))=(k-2) \cdot \frac{2}{3k-2}+2\cdot \frac{k+2}{2(3k-2)}=1$.
If $L_i$ is small, then it is the first or last layer in its segment.  As first and last layers in a segment of any type have weight at most $\frac{k}{3k-2}$,  we have $u(N(v))\le\frac{k}{3k-2}+(k-1)\cdot \frac{2}{3k-2}=1$.

Assume next that $v\in L_i$, where $L_i$ is  of Type 2. 
If $L_i$ is  small, and is not the first or last layer in the segment, then $u(N(v))= 2\cdot(k-1)\cdot \frac{1}{2k-1} =1$. If $L_i$ is the first or last layer in the segment, then
$u(N(v))\le  (k-1)\cdot \frac{1}{2(k-1)}+\frac{k}{3k-2}<1$. If $L_i$ is big, then $u(N(v))$ is the greatest if $L_i$ is the second or next-to-last layer in its segment. Therefore
\begin{eqnarray*}
u(N(v))&\le& \Biggl(\frac{1}{2}-\frac{1}{2(k-1)}\Biggl) +\frac{k+2}{2(3k-2)}+ \frac{3k+2}{4(3k-2)}=\frac{(11k+2)(k-1)-6k+4}{4(k-1)(3k-2)}\\
&=&\frac{(11k-4)(k-1)-2}{4(3k-2)(k-1)}\le 1.
\end{eqnarray*}

Assume that $v\in L_i$,   where $L_i$ is  of Type 3. Then $\max(u(L_i),u(L_{i-1}),u(L_{i+1}))\le\frac{k}{3k-2}$. If $u(v)\ge\frac{2}{3k-2}$, then $u(N(v))\le u(L_{i-1})+u( L_{i})+u( L_{i+1})-u(v)\le\frac{3k}{3k-2}-\frac{2}{3k-2}=1$.
Otherwise, we have that $|L_i|>\frac{k}{2}$ and $v\notin S_i$. Since $v\notin S_i$, there is a $j\in\{i-1,i+1\}$ and a $w\in L_j$ that is not a neighbor of $v$. 
As $w\in L_j\setminus S_j$, by Lemma~\ref{lm:basic}~\ref{part:small} we have that $|L_j|>1$, therefore $L_j$ is also of Type 3.

It  $u(v)+u(w)\ge\frac{2}{3k-2}$, then we get, as before, that
$u(N(v))\le 1$, as needed. In particular, if $u(w)\ge\frac{2}{3k-2}$ then we are done. So we may further assume that $|L_j|>\frac{k}{2}$ and $w\notin S_j$, 
Moreover, from Lemma~\ref{lm:basic}\ref{part:type3} we have $\max(|L_i|-|S_i|,|L_j|-|S_j|)\le k- |S_i|-|S_j|$. Therefore
\begin{eqnarray*}
u(v)+u(w)&=&\frac{k-2|S_i|}{(3k-2)(|L_i|-|S_i|)}+\frac{k-2|S_j|}{(3k-2)(|L_j|-|S_j|)}\\
&\ge&\frac{k-2|S_i|}{(3k-2)(k-|S_i|-|S_j|)}+\frac{k-2|S_j|}{(3k-2)(k-|S_i|-|S_j|)}
=\frac{2}{3k-2}.
\end{eqnarray*}
This finishes the proof.


\begin{thebibliography}{00}

\bibitem{Amar} 
    D. Amar, I. Fournier, A. Germa. 
    Ordre minimum d'un graphe simple de diam\`etre, degr\'{e} minimum et connexit\'{e} donn\'{e}s. 
    \textit{Ann Discrete Math.}, \textbf{17} (1983), 7--10.

\bibitem{Smyth} 
    L. Cacetta, W.F. Smyth.
    Graphs of maximum diameter. 
    \textit{Discrete Math.}, \textbf{102} (1992), 121--141.


\bibitem{dankelmanos}
    \'E. Czabarka, P. Dankelmann, L. A. Sz\'ekely. 
    Diameter of 4-colorable graphs. 
    \textit{Europ. J. Comb.} {\bf 30} (2009), 1082--1089. 
    
\bibitem{counterexpaper}
    \'E. Czabarka, P. Dankelmann, L. A. Sz\'ekely. 
      Counterexamples to a conjecture of Erd\H{o}s, Pach, Pollack and Tuza.
      \textit{ J. Combin. Theory} {\bf B  151} (2021), 38--45.
  
\bibitem{kcolorable}
    \'E. Czabarka, P. Dankelmann, L. A. Sz\'ekely. 
      On the diameter of $k$-colorable graphs.
      \textit{Electronic J. of Comb}, {\bf 28}(3) (2021) P3.52 
 

\bibitem{EPPT} 
    P. Erd\H{o}s, J. Pach, R. Pollack, and Z. Tuza.
    Radius, diameter, and minimum degree.
    \textit{ J. Combin. Theory} {\bf B 47} (1989), 279--285.
    
    
\bibitem{Gold} 
    D. Goldsmith, B. Manvel, V. Faber. 
    A lower bound for the order of the graph in terms of the diameter and minimum degree. 
    \textit{J. Combin. Inform. Syst. Sci.} \textbf{6} (1981), 315--319.
   
    
\bibitem{Moon} J.W. Moon. 
    On the diameter of a graph. 
    \textit{Mich. Math. J.} \textbf{12} (1965), 349--351.


\end{thebibliography}
\end{document}